\documentclass[12pt]{article}
\usepackage{amsmath, amsfonts, amssymb, latexsym, wasysym, graphicx, mathtools,  wasysym, enumitem, pifont}

\usepackage[all]{xy}

\usepackage{hyperref}
\hypersetup{colorlinks,citecolor=blue,linkcolor=blue}

\title{Stallings' Group is Simply Connected at Infinity}
\author{Michael Mihalik}

\newtheorem{theorem}{Theorem}[section]
\newtheorem{proposition}[theorem]{Proposition}
\newtheorem{lemma}[theorem]{Lemma}

\newtheorem{corollary}[theorem]{Corollary}

\newtheorem{conjecture}[theorem]{Conjecture}

\newenvironment{proof}{\addvspace{12pt}\noindent{\bf Proof:}}{
$\Box$\par\addvspace{12pt}}

\numberwithin{equation}{section}

\newcounter{definitionnum}
\newenvironment{definition}{\addvspace{12pt}\refstepcounter{definitionnum}
\noindent{\bf Definition \arabic{definitionnum}.}}{\par\addvspace{12pt}}

\date{\today}

\begin{document}

\maketitle

\begin{abstract} 
Let $F_2$ be the free group on two generators and let $B_n$ ($n \geq 2$) denote the kernel of the homomorphism 
$$F_2 \times  \cdots (n) \cdots \times F_2 \rightarrow {\mathbb Z}$$
sending all generators to the generator $1$ of $\mathbb Z$. The groups $B_k$ are called the {\it Bieri-Stallings} groups and $B_k$ is type $\mathcal F_{k-1}$ but not $\mathcal F_k$. For $n\geq 3$ there are short exact sequences of the form 
 $$1 \rightarrow B_{n-1} \rightarrow B_n \rightarrow F_2 \rightarrow 1.$$ 
 This exact sequence can be used to show that $B_n$ is $(n-3)$-connected at infinity for $n\geq 3$. Stallings' proved that $B_2$ is finitely generated but not finitely presented. We conjecture that for $n\geq 2$, $B_n$ is $(n-2)$-connected at infinity. For $n=2$, this means that $B_2$ is 1-ended and for $n=3$ that $B_3$ (typically called Stallings' group) is simply connected at infinity. We verify the conjecture for $n=2$ and $n=3$. Our main result is the case $n=3$:
Stalling's group is simply connected at $\infty$. 
\end{abstract}

\maketitle

\section{Introduction}

A group $G$ is of type $\mathcal F_1$ when it can be finitely generated, $\mathcal F_2$ when it can be finitely
presented, and more generally $\mathcal F_n$ when there is a $K(G,1)$-complex with finite $n$-skeleton. Let $F_2$ be the free group on two generators.

$$F_2\times F_2\equiv \langle a_1,a_2, b_1,b_2 \ |\  [a_i,b_j]=1\hbox{ for all } i,j\in \{1,2\}\rangle.$$ 
Stallings \cite{Stall63} showed that the kernel $K$ of the homomorphism $$f:F_2\times F_2\to \mathbb Z\hbox{ where } f(a_1)=f(a_2)=f(b_1)=f(b_2)=1,$$
is finitely generated but not finitely presented. Stallings' also proved that a certain HNN-extension of $K$ (now known as Stallings' group) is $\mathcal F_2$ but not $\mathcal F_3$. Notationally, Stallings' group is usually denoted as $S$ or $B_3$.  The group $K$ (above) is $B_2$. 

Let $B_n$ ($n \geq 2$) denote the kernel of the homomorphism 
$$F_2 \times  \cdots (n) \cdots \times F_2 \rightarrow {\mathbb Z}$$
 sending all generators to the generator $1$. (Here $K$ is isomorphic to $B_2$ and Stallings' group is isomorphic to $B_3$.)  Bieri \cite{Bieri81} proved that $B_n$ is $\mathcal F_{n-1}$ but not $\mathcal F_n$. These groups are often call the {\it Bieri-Stallings} groups. Gersten \cite{Ger95} proved that for $n\geq 3$, $B_n$ admits quintic isoperimetric function. For Stallings' group this was improved to cubic by Baumslag, Bridson, Miller and Short \cite{BBMS97} and then to quadratic by Dison, Elder, Riley and Young \cite{DERY09}. This last paper also shows that the asymptotic cones of Stallings' group are all simply connected, but are not all 2-connected. Finally, Stallings' group admits a linear isodiametric function.

 There are short exact sequences of the form
 $$1 \rightarrow B_{n-1} \rightarrow B_n \rightarrow F_2 \rightarrow 1.$$ 
 \medskip
 
\noindent [{\bf Theorem 17.3.6},\cite{G}]. {\it Let $1\to N\to G\to Q\to 1$  be a short exact sequence of infinite
groups of type $\mathcal F_n$, and let $R$ be a PID. Working with respect to $R$, let $N$ be
$s$-acyclic at infinity and let $Q$ be $t$-acyclic at infinity where $s\leq n-1$ and
$t \leq n-1$. Then $G$ is $u$-acyclic at infinity where $u=min\{s + t + 2, n-1\}$. If
$s\geq 0$ or $t\geq 0$  (i.e., if $N$ or $Q$ has one end) then $G$ is $u$-connected at infinity.}
 
 \medskip
 
 This theorem implies that $B_n$ is $(n-3)$-connected at infinity, for $n \geq 4$
 (since $B_{n-1}$ is type $\mathcal F_{n-2}$ and not $\mathcal F_{n-1}$, we cannot conclude that $B_n$ is $(n-2)$-connected at infinity.) The proof of our main theorem is geometric.
 
 \begin{theorem} \label{Main} {\bf (Main)} 
Stallings' group $B_3$ is simply connected (1-connected) at infinity.
\end{theorem} 

We also provide an elementary algebraic proof (Lemma \ref{OneE}) that $B_2$ is 1-ended. These two results are evidence (and the first two cases) of:

\begin{conjecture} \it For $n\geq 2$ the group $B_n$ is $(n-2)$-connected at infinity.
 \end{conjecture}
 We note that in order for a group to be $(k-1)$-connected at infinity, it must be of type $\mathcal F_{k}$ (see \S \ref{DN}) and so, if true, the Conjecture would be best possible. 
 Parts (ii) and (iii) of the main theorem of \cite{GM86} imply that if $G$ is a finitely presented simply connected at infinity group then $H^2(G,\mathbb ZG)=0$ (also see [Theorem 16.5.2, \cite{G}]). A direct consequence of this fact and Theorem \ref{Main} is:
 \begin{corollary} For Stallings' group $S$,  $H^2(S,\mathbb ZS)=0$.
 \end{corollary}
\noindent This article is organized as follows. Section \ref{DN} contains the basic definitions and notation used in our proofs. In particular, we discuss the idea of a finitely generated group being simply connected at infinity in a finitely presented group as well as the idea of a finitely generated group being simply connected at infinity. These are fundamentally important ideas in the proof of our lemmas and main theorem. 
Section \ref{StG} describes presentations of the groups involved as well as the spaces we work in. Section \ref{vKD} contains an elementary description of van Kampen diagrams and $s$-bands. These are primary tools used to prove some of our lemmas. 
Section \ref{L} contains proofs of our lemmas. We prove that the kernel $K$ of the map $F_2\times F_2\to \mathbb Z$ described in the introduction is 1-ended (Lemma \ref{OneE}) and construct useful generating sets for this group. Stallings' group $S$ is an HNN extension of $F_2\times F_2$ over $K$. We use van Kampen diagrams connected to this HNN-extension, to reduce the problem of showing that $S$ is simply connected at infinity to showing that $F_2\times F_2$ is simply connected at infinity in $S$ (Lemma \ref{Reduce}). Lemma \ref{F2P} basically reduces this problem to showing that $K$ is simply connected at infinity in $S$. 
Finally, Section \ref{SCI} combines the lemmas of the previous sections to prove that $S$ is simply connected at infinity. 

\noindent {\it Acknowledgements:} We thank Francisco Lasheras for bring the question of the simple connectivity at infinity of Stallings' group to our attention and for many helpful conversations.
 \section{Definitions and Notation} \label{DN}
Let $D^k$ be the {\it $k$-disk} in $\mathbb R^{k}$ (all points of distance $\leq 1$ from the origin) and $S^{k-1}$ the {\it $(k-1)$-sphere} in $\mathbb R^k$ (all points of distance $1$ from the origin). Let $n$ be an integer $\geq 1$. A path connected space $X$ is {\it $n$-connected}  ({\it $n$-aspherical}) if for every $1\leq k\leq n$ (respectively $2\leq k\leq n$), every map $S^k\to X$ extends to a map of $D^k\to X$. A CW-complex $X$ is {\it $n$-connected} ($n$-aspherical) if and only if $X^{n+1}$ (its $(n+1)^{st}$ skeleton) is $n$-connected ($n$-aspherical). A space $X$ is {\it aspherical} if $X$ is $n$-aspherical for all $n$. 

\begin{proposition} [Proposition 7.1.3,\cite{G}] 
A path connected CW complex $X$ is $n$-aspherical if and only if its universal cover $\tilde X$ is $n$-connected. The complex $X$ is aspherical if and only if $\tilde X$ is contractible. 
\end{proposition}

A space $Y$ is {\it $n$-connected at infinity} if for any compact $C\subset Y$ there is a compact $D\subset Y$ such that any map of $S^n\to Y-D$ can be extended to a map $D^{n+1}\to X-C$. Let $G$ be a group. A {\it $K(G,1)$-complex} is a path connected aspherical CW-complex $X$ with $\pi_1(X)$ isomorphic to $G$. The space $X$ is also called an {\it Eilengerg-MacLane complex of type $(G,1)$} or a {\it classifying space for $G$}.  A group $G$ has type $\mathcal F_n$ if there exists a $K(G,1)$-complex having finite $n$-skeleton.  Let $G$ be a group of type $\mathcal F_n$ where $n\geq 1$ and let $X$ be a $K(G,1)$-complex with finite $n$-skeleton. We say that $G$ is {\it $(n-1)$-connected at infinity} if $X^n$ (the universal cover of the $n$-skeleton of $X$) is $(n-1)$-connected at infinity. In particular:

\begin{definition} \label{SCinf} 
The one ended finitely presented group $G$ is {\it simply connected at infinity (1-connected at infinity)} if for some (any) finite connected complex $A$ with $\pi_1(A)$ isomorphic to $G$, the universal cover $\tilde A$ has the following property: For any finite subcomplex $C\subset \tilde A$ there is a finite subcomplex $D\subset \tilde A$ such that loops in $\tilde A-D$ are homotopically trivial in $\tilde A-C$. 
\end{definition}

The following definition explains what it means for a finitely generated subgroup of a finitely presented group $G$ to be simply connected at $\infty$ in $G$. This notion was first considered in \cite{M6} and has been used successfully to prove interesting classes of finitely presented groups are simply connected at $\infty$. This notion is fundamental in our proof that Stallings' group is simply connected at infinity. 

\begin{definition}\label{SCin} 
A finitely generated subgroup $A$ of a finitely presented group $G$ is {\it simply connected at $\infty$ in $G$ (or relative to $G$)} if for some (equivalently any by [Lemma 2.9, \cite{M6}])  finite presentation $\mathcal P=\langle \mathcal A,\mathcal B\ | \ R\rangle$ of the group $G$  (where $\mathcal A$ generates $A$ and $\mathcal A\cup \mathcal B$ generates $G$), the Cayley 2-complex $\Gamma(\mathcal P)$ has the following property:

Given any compact set $C\subset \Gamma ( \mathcal P)$ there is a compact set $D\subset \Gamma ( \mathcal P)$ such that any edge path loop in  $\Gamma (A,\mathcal A)-D$ is homotopically trivial in $\Gamma ( \mathcal P)-C$. 
\end{definition}
In order to define what  it means for a finitely generated group $G$ to be simply connected at $\infty$, we must know that $G$ embeds in some finitely presented group. In 1961, G. Higman proved:

\begin{theorem} [Higman, \cite{Hig61}]  A finitely generated infinite group $G$ can be embedded in a finitely presented group if and only if  $G$ is recursively presented.
\end{theorem}

\begin{definition}\label{SCfg} 
A finitely generated and recursively presented group $A$ is {\it simply connected at $\infty$} if for any finitely presented group $G$ and subgroup $A'$ isomorphic to $A$, the subgroup $A'$ is simply connected at $\infty$ in $G$. 
\end{definition}

We will use the following result in $\S$\ref{SCI}.

\begin{theorem} [Theorem 5.1, \cite{M6}] \label{sc}
Suppose the recursively presented group $G$ is finitely generated and isomorphic to $A\times B$ where $A$ and $B$ are finitely generated infinite groups and $A$ is 1-ended. Then $G$ is simply connected at $\infty$.
\end{theorem}
 
  For any set $A$, $F(A)$ is the free group on $A$. If $\mathcal A$ is a finite generating set for a group $G$, let $\Gamma(G,\mathcal A)$ be the Cayley graph of $G$ with respect to $\mathcal A$. So the vertex set of $\Gamma(G,\mathcal A)$ is $G$ and there is an edge labeled $a\in \mathcal A$ from $v$ to $w$ if $va=w$.
Given a finite presentation $\mathcal P=\langle \mathcal A : R\rangle$ of a group $G$, the {\it Cayley 2-complex for $\mathcal P$} is a 2-complex $\Gamma(\mathcal P)$ with 1-skeleton equal to  $\Gamma (G,\mathcal A)$. At each vertex $v$ of $\Gamma(\mathcal P)$ there is a 2-cell with edge labeling (beginning at $v$) equal to an element of $R$. The space $\Gamma(\mathcal P)$ is simply connected, the group $G$ acts as covering translations on $\Gamma(\mathcal P)$ and the quotient space $\Gamma(\mathcal P)/G$ (the {\it presentation complex for $\mathcal P$}) has fundamental group isomorphic to $G$. If $\mathcal B\subset \mathcal A$ and $B$ is the subgroup of $G$ generated by $\mathcal B$ then $\Gamma(B, \mathcal B)\subset \Gamma(G,\mathcal A)\subset \Gamma(\mathcal P)$. If $\mathcal Q$ is a subpresentation of $\mathcal P$ for a subgroup $Q$ of $G$ then $\Gamma(\mathcal Q)\subset \Gamma(\mathcal P)$. For each $g\in G$, $g\Gamma(B,\mathcal B)\subset \Gamma(G,\mathcal A)$ and $g\Gamma(\mathcal Q)\subset \Gamma(\mathcal P)$. 

\newpage

\section{Stallings' Group and Related Spaces}\label{StG} 
S. Gersten [Proposition 2.3, \cite{Ger95}] exhibits the following presentation of Stallings' group. This is also the presentation used in \cite{BBMS97} and \cite{DERY09}. Here $s^k$ means $k^{-1}sk$ and $[a,b;c,d]$ means both $a$ and $b$ commute with both $c$ and $d$.  
$$S=\langle a,b,c,d,s: [a,b;c,d], s^a=s^b=s^c=s^d\rangle.$$
Expanding our presentation of $S$ we have $b^{-1}s^{-1}ba^{-1}sa=1$. Equivalently, $s^{-1}ba^{-1}s=ba^{-1}$. Similar expansions imply that the latter defining relations of our presentation of $S$ are equivalent to the commutation relations:
$$[s;ba^{-1}, ca^{-1}, da^{-1}, cb^{-1}, db^{-1}, dc^{-1}].$$

Hence, $S$ is an HNN extension (with stable letter $s$) of the 1-ended group $F(a,b)\times F(c,d)$ over the group $K$ generated by the six elements
$$e_1=ba^{-1}, e_2=ca^{-1}, e_3=da^{-1},  e_4=cb^{-1}, e_5=db^{-1}, e_6=dc^{-1}$$
where $s$ commutes with each of these elements. We will show (Lemma \ref{NFP}) that the subgroup $K$ of $F(a,b)\times F(c,d)$ generated by $\{e_1,\ldots, e_6\}$ is normal and hence equal to the kernel of the map $F(a,b)\times F(c,d)\to \mathbb Z$ that takes each element of $\{a,b,c,d\}$ to $1$ (a generator of $\mathbb Z$). In particular, $K$ consists of all words in $\{a,b,c,d\}^{\pm 1}$ with 0-exponent sum. Note that $K$ has (finite) generating set consisting of all two letter words of the form $uv^{-1}$ or $u^{-1}v$ for $u,v$ distinct in $\{a,b,c,d\}$. List these words as $w_1,\ldots, w_q$.

In all that follows we will use the following presentation for $S$ 
$$\mathcal S=\langle a,b,c,d,s,e_1,\ldots, e_q: e_1=w_1, \ldots, e_q=w_q, [a,b ; c,d], [s ; e_1,\ldots, e_q ]\rangle.$$ 
Consider the following presentations for $F(a,b)\times F(c,d)$:
$$\mathcal Q_1=\langle a,b,c,d : [a,b ; c,d]\rangle,$$
$$\mathcal Q_2=\langle a,b,c,d, e_1,\ldots, e_q:[a,b;c,d], e_1=w_1, \ldots, e_q=w_q\rangle.$$ 
Let $X=\Gamma(\mathcal S)$ (with vertex set $S$), $\Gamma_1$ is the Cayley graph of $F(a,b)\times F(c,d)$ with respect to $\{a,b,c,d\}$ and $\Gamma_2$ is the Cayley graph of $F(a,b)\times F(c,d)$ with respect $\{a,b,c,d,e_1,\ldots, e_q\}$. Let $E=\{e_1,\ldots, e_q\}$ and $\Gamma(K)=\Gamma(K, E)$. Let $\ast$ be the identity vertex of $X$ then:  
$$\ast\in\Gamma(K)\subset \Gamma_2\subset\Gamma(\mathcal Q_2)\subset X \hbox{ and } \ast\in \Gamma_1\subset \Gamma(\mathcal Q_1)\subset\Gamma(\mathcal Q_2).$$ 

If $A$ is a subcomplex of $X$, then let $N(A)$ be the full subcomplex of $X$ with vertex set equal to all vertices of $A$ plus all vertices an edge away from a vertex of $A$. 

\section{van Kampen Diagrams and $s$-Bands}\label{vKD} 

Suppose $G$ has a finite presentation $\mathcal P=\langle S:R\rangle$. This means that there is a epimorphism of the free group $q:F(S)\to G$  and the kernel  of $q$ is $N(R)$, the normal closure of $R$ in $F(S)$. First observe that each element of $N(R)$ is a product of conjugates of elements of $R^{\pm1}$. Each non-trivial element of $F(S)$ is a product of elements of $S^{\pm 1}$ called an {\it $S$-word}.  Suppose $w$ is a reduced $S$-word (no occurrence of $aa^{-1}$ (or $a^{-1}a$) in $w$) in the kernel of $q$, then $w=(u_1^{-1}r_1u_1)\cdots (u_n^{-1}r_nu_n)$ where $u_i\in F(S)$ and $r_i\in R^{\pm 1}$ for all $i$. This means that if all occurrences of $aa^{-1}$ (or $a^{-1}a$) in $(u_1^{-1}r_1u_1)\cdots (u_n^{-1}r_nu_n)$ are repeatedly eliminated, we are left with the word $w$. Clearly we can draw (see Figure \ref{Fig1}) an edge path loop with edge labeling determined by the $S$-word $(u_1^{-1}r_1u_1)\cdots (u_n^{-1}r_nu_n)$ in the plane (where the subpath for each $r_i$ is a loop). 
\begin{figure}
\vbox to 3in{\vspace {-2in} \hspace {-.7in}
\hspace{-1 in}
\includegraphics[scale=1]{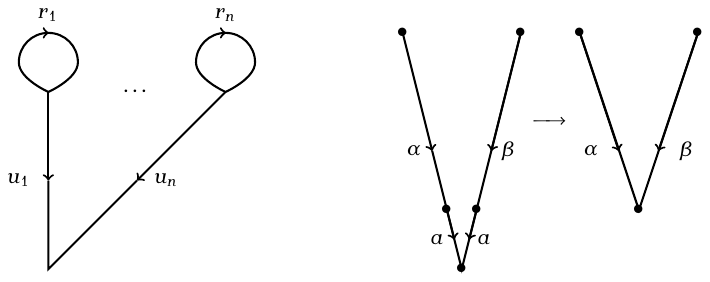}
\vss }
\vspace{-1.5in}
\caption{van Kampen Diagrams and Collapsing Edges} 
\label{Fig1}
\vspace{-.1in}
\end{figure}
Collapsing all adjacent edges labeled $aa^{-1}$ (or $a^{-1}a$)  gives a loop in the plane with edge labeling equal to $w$. Each region inside this loop is bounded by a loop with label in $R^{\pm 1}$. This labeled graph is a van Kampen diagram for $w$. 

\begin{lemma}\label{map} 
Suppose  $\mathcal P=\langle S:R\rangle$ is a finite presentation of a group $G$. Given a  van Kampen diagram $K$ for an $S$-word $w$ and a vertex $v$ of $\Lambda$ (the Cayley graph of $G$ with respect to the generating set $S$) there is map of $K$ to $\Lambda$ respecting edge labels such that the initial vertex of $w$ is mapped to $v$. 
\end{lemma}
\begin{proof}
Certainly the uncollapsed wedge of loops of Figure \ref{Fig1} can be mapped into $\Lambda$ respecting edge labels. Collapsing edges corresponds to removing backtracking pairs of edges.
\end{proof}
The image of a van Kampen diagram $\mathcal D$ in $\Lambda$ is called the {\it realization of} $\mathcal D$ in $\Lambda$. 
Notice that the edge path loop of the uncollapsed wedge of loops of Figure \ref{Fig1} traverses each edge exactly once or twice. After collapsing this is still the case for the resulting path, the {\it boundary path}, of the van Kampen diagram. 
Suppose $\mathcal D$ is a van Kampen diagram for a relation $r$ of the presentation $\mathcal S$ of Stallings' group $S$. Notice that the only relations of $\mathcal S$ involving the letter $s$ have the form $[s;e_i]$. If an edge $k$ with label $s$ appears in the boundary of $\mathcal D$ then it is paired with another edge $k'$ in the following way: If the boundary path of $\mathcal D$ (labeled by $r$) crosses $k$ twice (in opposite directions) then $k$ is paired with itself. If $k$ is only crossed once then it belongs to a cell (a square) $C_1$ with labeling $[s;e_i]$. If the $s$-edge of this square opposite $k$ belongs to another square $C_2$, then continue selecting such squares until a square $C_n$ is obtained with opposite $s$-edge in the boundary of $\mathcal D$. This last edge is paired with $k$. The collection of squares $C_1,\ldots, C_n$ form a rectangle $B$ in the plane called a {\it band} for $k$ with {\it top} edge $k$ (labeled $s$) and bottom edge $k'$ (labeled $s$). The two {\it sides} of this band have the same labeling $(t_1,\ldots, t_n)$ where each $t_i\in \{e_1,\ldots, e_q\}^{\pm 1}$.  Notice that two bands cannot cross one another and a band cannot cross itself (since no side edge of a band is labeled $s$). See Figure \ref{Fig2}. The image of a band in $\Lambda$ is called the {\it realization} of the band.
\begin{figure}
\vbox to 3in{\vspace {-2in} \hspace {-.7in}
\hspace{-.5 in}
\includegraphics[scale=1]{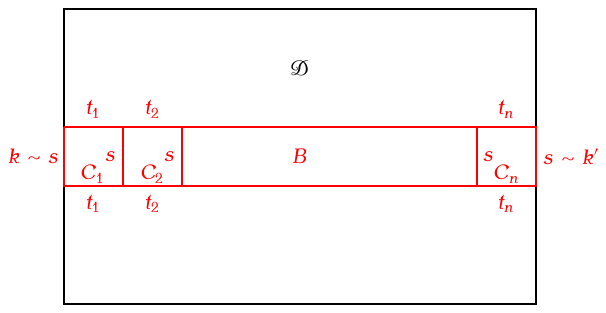}
\vss }
\vspace{-1.5in}
\caption{Bands in van Kampen Diagrams} 
\label{Fig2}
\vspace{-.1in}
\end{figure}

\section{The Lemmas}\label{L}
Our goal is to prove that $X$ is simply connected at infinity. All homotopies between two paths in $X$ with the same end points will be homotopies relative to those end points. 
\begin{lemma}\label{NFP}
The group $K$ is normal in $F(a,b)\times F(c,d)$ and hence it consists of all elements of $F(a,b)\times F(c,d)$ with zero exponent sum and it is the kernel of the map of $F(a,b)\times F(c,d)\to \mathbb Z$ where all generators are mapped to the generator $1$ of $\mathbb Z$.
\end{lemma} 

\begin{proof}
The group $K$ is generated by the six element set
$$T=\{ba^{-1}, ca^{-1}, da^{-1}, cb^{-1}, db^{-1}, dc^{-1}\}.$$
First we show all $0$-exponent sum 2-letter words in $\{a,b,c,d\}^{\pm 1}$ are in $K$.  Since $a$ and $c$ commute, all $0$-exponent sum 2-letter words in $\{a,c\}^{\pm 1}$ are in $K$. Similarly for $\{a,d\}^{\pm 1}$, $\{b, c\}^{\pm 1}$, and $\{b, d\}^{\pm 1}$. Next we consider the four $0$-exponent sum 2-letter words in $\{a,b\}^{\pm 1}$. Certainly $ba^{-1}\in K$ and so $ab^{-1}=(ba^{-1})^{-1} \in K$.   Observe that $b^{-1}a=(b^{-1}c)(c^{-1}a)\in K$, and $(b^{-1}a)^{-1}=a^{-1}b\in K$. Similarly for $\{c,d\}^{\pm 1}$, and all $0$-exponent sum 2-letter words in $\{a,b,c,d\}^{\pm 1}$ are in $K$.  
In order to see that $K$ is normal in $F(a,b)\times F(c,d)$ it is enough to show that each conjugate of an element of $T$ by an element of $\{a,b,c,d\}^{\pm 1}$ is in $K$. Consider all such conjugates of $ba^{-1}$. 
Note that $a(ba^{-1})a^{-1}=aba^{-2}=(ac^{-1})(bc^{-1})(ca^{-1})(ca^{-1})(\in K)$, $a^{-1}(ba^{-1})a=a^{-1}b$, $b^{-1}(ba^{-1})b=a^{-1}b$, $b(ba^{-1})b^{-1}=b^2a^{-1}b^{-1}=(bc^{-1})(bc^{-1})(ca^{-1})(cb^{-1})\in K$. Next observe that $a(ca^{-1})a^{-1}=ca^{-1}\in K$, $b(ac^{-1})b^{-1}=(bc^{-1})(ab^{-1})\in K$. Similarly all single letter conjugates of the other elements of $T$ are in $K$.   
\end{proof}

\begin{lemma} \label{OneE} 
The group $K$ is 1-ended.
\end{lemma}
\begin{proof}
Recall that the set $T=\{ba^{-1}, ca^{-1}, da^{-1}, cb^{-1}, db^{-1}, dc^{-1}\}$ generates $K$, 
Since $(cb^{-1})(ba^{-1})=ca^{-1}$ and $(dc^{-1})^{-1}(db^{-1})=cb^{-1}$, the group $K$ is generated by $\{ba^{-1}, da^{-1}, db^{-1},  dc^{-1}\}$. Finally, $(da^{-1})(ba^{-1})^{-1}=db^{-1}$ so that the commuting elements $ba^{-1} , dc^{-1}$ along with the element $da^{-1}$ generate $K$. 
$$K=\langle ba^{-1}, dc^{-1}, da^{-1}\rangle$$
Since $F(a,b)\times F(c,d)$ is torsion free, $K$ is torsion free. If $K$ has more than one end, then $K=A\ast B$ where $A$ and $B$ are infinite finitely generated groups. The subgroup $\langle ba^{-1}, dc^{-1}\rangle$ of $K$ is isomorphic to $\mathbb Z\times \mathbb Z$. One ended subgroups of $K$ lie in a conjugate of $A$ or $B$. Say $\langle ba^{-1}, dc^{-1}\rangle$ lies in a conjugate of $B$.  Say $\langle ba^{-1}, dc^{-1}\rangle \subset u^{-1}Bu$. As $K=u^{-1}Ku=(u^{-1}Au)\ast (u^{-1}Bu)$, we may assume that $\langle ba^{-1}, dc^{-1}\rangle\subset B$. 
Note that any conjugate of $ba^{-1}$ by an element of $F(a,b)\times F(c,d)$ is an element of $K\cap F(a,b)$ and hence it commutes with $dc^{-1}$. In particular, if $v$ is a conjugate of $ba^{-1}$ then $\langle v, ba^{-1}, dc^{-1}\rangle$ is a 1-ended subgroup of $K$ and so it lies in a $K$-conjugate of $B$. If two conjugates of $B$ in $A\ast B$ intersect non-trivially, then they are the same (Corollary 4.1.5 of \cite{MKS04}). Hence all conjugates of $ba^{-1}$ by elements of $F(a,b)\times F(c,d)$ lie in $B$. Similarly for $dc^{-1}$. This means that the normal closure of $\{ba^{-1}, dc^{-1}\}$ in $K$ is a subset of $B$. But if $x$ is a non-trivial element of $A$ then $xba^{-1}x^{-1}\not \in B$ (since it has syllable length 3 in the free product $A\ast B$). Instead, $K$ is 1-ended. 
\end{proof}
The next lemma reduces our problem to considering $\Gamma_1$-loops in $X$. 
\begin{lemma}\label{Reduce} 
If $F(a,b)\times F(c,d)$ is simply connected at infinity in $S$, then $X$ is simply connected at infinity.
\end{lemma}
\begin{proof}
Assume $F(a,b)\times F(c,d)$ is simply connected at infinity in $S$. 
Let $C$ be a finite subcomplex of $X$. For $k\in S$, there are only finitely many $k\Gamma(\mathcal Q_2)$ that intersect $C$. Choose $k_1,\ldots, k_n$ such that if $k\Gamma(\mathcal Q_2)\cap C\ne\emptyset$, then $k\Gamma(\mathcal Q_2)\in \{k_1\Gamma(\mathcal Q_2),\ldots, k_n\Gamma(\mathcal Q_2)\}$. Since $F(a,b)\times F(c,d)$ is simply connected at infinity in $S$ there is a finite complex $D_i\subset X$ such that if $\beta$ is an edge path loop in $k_i\Gamma(\mathcal Q_2)-D_i$, then $\beta$ is homotopically trivial in $X-C$. Let $D=C\cup (\cup_{i=1}^nD_i$). If $\beta$ is an edge path loop in $g\Gamma(\mathcal Q_2)$ for $g\not\in \{k_1,\ldots, k_n\}$ then since $\Gamma(\mathcal Q_2)$ is simply connected, $\beta$ is homotopically trivial in $g\Gamma(\mathcal Q_2)$ and so by a homotopy avoiding $C$. We have:

\medskip
 
\noindent ($\ast$) If $\beta$ is an edge path loop in $k\Gamma(\mathcal Q_2)-D$ for any $k\in S$, then $\beta$ is homotopically trivial in $X-C$. 

\medskip

For $g\in S$, there are only finitely many $g\Gamma(K)$ in $k_i\Gamma(\mathcal Q_2)$ that intersect $D_i$. List them as $g_j\Gamma(K)$ for $j\in J_i$. The group $K$ is 1-ended and so, for each $j\in J_i$, there is a finite subcomplex $A_{g_j}$ of $g_j\Gamma(K)$ such that any two vertices in $g_j\Gamma(K)-A_{g_j}$ can be joined by an edge path in $g_j\Gamma(K)-N(D_i)$. Let $A _i =N(D_i)\cup (\cup_{j\in J_i} A_{g_j})$. 
Let $A=\cup_{i=1}^n A_i$. Suppose that $\alpha$ is an edge path loop in $X-A$ our goal is to show that $\alpha$ is homotopically trivial in $X-C$. 
Let $\bar \alpha$ be the $S$-word labeling the edges of $\alpha$.
 By eliminating backtracking in $\alpha$, we may assume that $\bar \alpha$ is a reduced $S$-word that labels $\alpha$. Choose a van Kampen diagram $\mathcal D$ for $\bar \alpha$ (arising from the presentation $\mathcal S$). If an $s$-letter of the boundary of $\mathcal D$ is paired with itself, there is nothing to do. Otherwise choose a band for that $s$-letter. Each side of the band is labeled by a word  in $\{e_1,\ldots, e_q\}^{\pm 1}$ (defining an element of $K$) and both sides have the same labeling (see Figure \ref{Fig2}). If $\alpha$ has no $s$-edge we are finished by ($\ast$). Suppose $v$ is the initial vertex of an $s$-edge $e$ of $\alpha$ where $e$ is paired with another edge. Let $B$ be the realization of a band for $e$ in $X$. Let $w$ be the initial vertex of the corresponding $s$-edge $\bar e$ (at the end of $B$). Let $\gamma$ be the $\{e_1,\ldots e_q\}^{\pm 1}$ edge path bounding $B$ from $v$ to $w$. 

By the selection of $A$ (and the 1-endedness of $K$) we may choose an $\{e_1,\ldots, e_q\}^{\pm 1}$ edge path $\delta_1$  from $v$ to $w$ that avoids $N(D)$. 
Let $\bar B$ be the realization of a band in $X$ for $e$ and $\bar e$ with one side equal to $\delta_1$ and the other side (with the same labeling as $\delta_1$) equal to $\delta _2$. Since $\delta_1$ avoids $N(D)$, this band avoids $D$. In particular $\delta_2$ avoids $D$ and $\delta_1$ is homotopic  
to $(e,\delta_2, \bar e^{-1}$) by a homotopy $H_{(e,\bar e)} $(via $\bar B$) in $X-C$. 

\begin{figure}
\vbox to 3in{\vspace {-2in} \hspace {.5in}
\hspace{-2.5 in}
\includegraphics[scale=1]{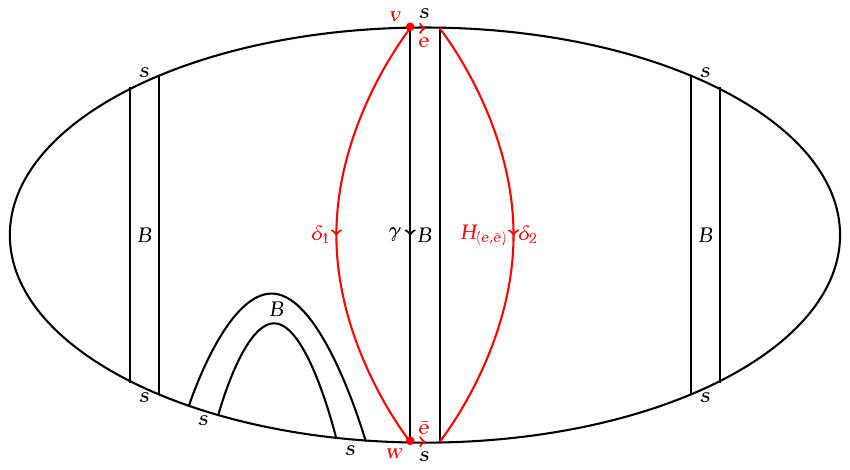}
\vss }
\vspace{-.4in}
\caption{Eliminating Bands} 
\label{FigR}
\vspace{.3in}
\end{figure}

Recall that $s$-bands cannot cross themselves or one another. Hence when we produce such homotopies for each such pair of $s$-edges, we can complete a null homotopy for $\alpha$ in $X-C$ by finding homotopies for the loops of Figure \ref{FigR} comprised of segements of $\alpha$ with no $s$-edges and $\delta_i$ type segments. Each of these loops belongs to $k\Gamma(\mathcal Q_2)-D$ for some $k\in S$. By ($\ast$) these loops are homotopically trivial in $X-C$. Combining these homotopies with the $H_{(e,\bar e)}$ homotopies (see Figure \ref{FigR}) creates a null homotopy for $\alpha$ in $X-C$. 
\end{proof}

Suppose $v$ is a vertex of $\Gamma_1$ (the Cayley 2-complex of the group $F(a,b)\times F(c,d)$ with respect to $\{a,b,c,d\}$) then there is a unique geodesic edge path $(\alpha, \beta)$ in $\Gamma_1$ from the identity vertex $\ast$ to $v$ where the labeling of $\alpha$ and $\beta$ are geodesics in $F(a,b)$ and $F(c,d)$ respectively. Call this path {\it the $F(a,b)-F(c,d)$ geodesic}.  Note that the (unique) edge path $(\beta', \alpha')$ at $\ast$ where $\beta'$ and $\alpha'$ have the same labeling as $\beta$ and $\alpha$ respectively also defines a geodesic in $\Gamma_1$ to $v$. Call this path {\it the $F(c,d)-F(a,b)$ geodesic}. For $v\in \Gamma_1$ and integer $n\geq 0$, let $N(v,n)$ be the ball of radius $n$ about $v$ in $\Gamma_1$ so $N(v,n)$ is the full subgraph of $\Gamma$ with vertices the set of all vertices $w$ such that there is a geodesic edge path of length $\leq n$ from $v$ to $w$. Call an edge path in $\Gamma_1$ with all edge labels $t\in Q\subset \{a,b,c,d\}^{\pm 1}$ a $Q$-path. Call an even length path in $\Gamma_1$ with alternating $\pm 1$ exponents a {\it $K$-path}. Suppose $\tau$ is an $\{a,b\}^{\pm 1}$ path at the vertex $v\in\Gamma_1$ and $\gamma$ is a $\{c,d\}^{\pm 1}$ path at $v$. Let $\gamma'$ ($\tau'$)  be the path at the end point of $\tau$ ($\gamma$) with the same labeling as $ \gamma$ ($\tau)$. There is a {\it product homotopy} between $(\tau, \gamma)$ and $(\gamma',\tau')$ (built from the commutation 2-cells of the relations $[a,b:c,d]$) in $v\Gamma(\mathcal Q_1)\subset X$. Each part of Lemma \ref {Ptrade} is clear. 

\noindent {\bf Note.} The exponents of the letters in $\tau$ are opposite to those of $\gamma$ in part (3), determining the $K$-path $\psi$.

\begin{lemma} \label{Ptrade} 
Suppose $v$ is a vertex of $\Gamma_1$ and $(\alpha,\beta)$ is the $F(a,b)-F(c,d)$ geodesic from $\ast$ to $v$. 

\noindent (1) If $(\hat\alpha,\hat\beta)$ is the $F(a,b)-F(c,d)$ geodesic to a point of $vF(a,b)$ then $\hat \beta$ has the same labeling as $\beta$. The $F(c,d)-F(a,b)$ geodesic to any point of $vF(c,d)$ has the form $(\alpha, \bar\beta)$. 

\medskip

\noindent (2) If the last edge label of $\alpha$ is $a$ and $\gamma$ is a geodesic $t$-path at $v$ for $t\in \{a,b,b^{-1}\}$ then $(\alpha, \beta,\gamma)$ is geodesic so that $((\alpha,\gamma'),\beta')$ is the $F(a,b)-F(c,d)$-geodesic from $\ast$ to the end point of $\gamma$, where $\gamma'$ (respectively $\beta'$) has the same labeling as does $\gamma$ (respectively $\beta$). In particular, if $v\not \in N(\ast, n)$ then no point of $\gamma$ is in $N(\ast, n)$ (in fact, $\gamma$ moves geodesically away from $\ast$).  Similarly if the last letter of $\alpha$ is $a^{-1}$, $b$ or $b^{-1}$ or the last letter of $\beta$ is $c$, $c^{-1}$, $d$ or $d^{-1}$. 
In particular, for any vertex $v$ in $\Gamma_1$ either every $a$-path at $v$  or every $b$-path at $v$ moves geodesically away from $\ast$. Similarly for $c$ and $d$ paths, $a^{-1}$ and $b^{-1}$-paths, and $c^{-1}$ and $d^{-1}$ paths at $v$.  

\medskip

\noindent (3) Suppose $\tau$ is an $\{a,b\}$-path of length $k$ in $\Gamma_1-N(\ast, m)$ from $v$ to $w$. Say the last letter of $\beta$ is $c$ or $c^{-1}$ (respectively $d$ or $d^{-1}$). If $\gamma$ is a $d^{-1}$ (respectively $c^{-1}$) path at $w$ (see the above Note), then $(\tau, \gamma)$ is homotopic to $(\gamma',\tau')$ (a path at $v$ where $\tau'$ and $\gamma'$ have the same labeling as $\tau$ and $\gamma$ respectively)  by a (product) homotopy in $\Gamma(\mathcal Q_1)-N(\ast,n)$ (statement (2) implies that each vertex of the product homotopy is in $\Gamma(\mathcal Q_1)-N(\ast,n)$). Furthermore, if $\gamma$ has length equal to the length of $\tau$ then the $K$-path $\psi$ at $v$ obtained by alternating letters of $\tau$ and $\gamma$ is homotopic to $(\tau,\gamma)$ by a homotopy with image in $\Gamma(\mathcal Q_1)-N(\ast,n)$. (Figure \ref{FigP}).

Similarly, if $\tau$ is an $\{a^{-1},b^{-1}\}$, $\{c,d\}$ or $\{c^{-1},d^{-1}\}$-path then there is a corresponding $(c)$ or $(d)$, $(a^{-1})$ or $(b^{-1})$, $(a)$ or $(b)$-path $\gamma$ respectively. $\square$
\end{lemma}

\begin{lemma} \label{F2P}
Suppose $\tau$ is an edge path in $\Gamma(\mathcal Q_1) -N(\ast, m)$ with exponent sum equal to $0$. Then $\tau$ is homotopic to a $K$-path by a homotopy in $\Gamma(\mathcal Q_1)-N(\ast, m)$. 
\end{lemma}

\begin{figure}
\vbox to 3in{\vspace {-2in} \hspace {.5in}
\hspace{-1.5 in}
\includegraphics[scale=1]{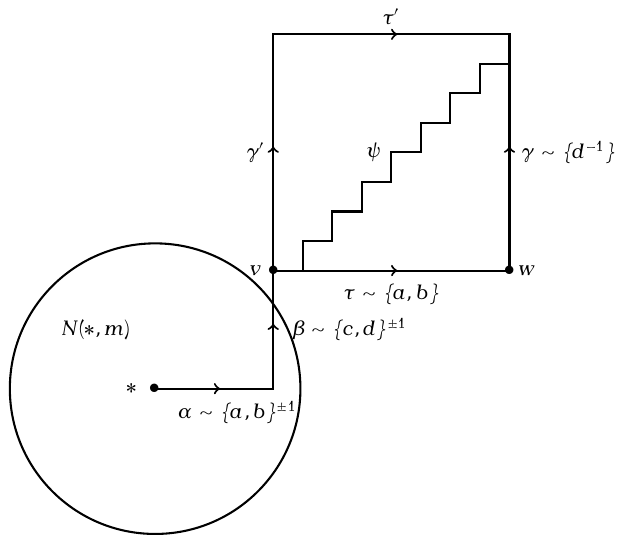}
\vss }
\vspace{.4in}
\caption{Product Homotopies in $\Gamma(\mathcal Q_1) -N(\ast, m)$} 
\label{FigP}
\vspace{.1in}
\end{figure}

\begin{proof} Write $\tau$ as $(\tau_1,\ldots, \tau_p)$ where each $\tau_i$ is a maximal $\{a,b\}$, $\{a^{-1},b^{-1}\}$, $\{c,d\}$ or $\{c^{-1},d^{-1}\}$-path connecting vertices $v_i$ and $v_{i+1}$. We call each $\tau_i$ a {\it syllable}. Let $k_i$ be the length of $\tau_i$. 
We prove this theorem by induction on $p$. Note that if $p=1$ then this path must be the trivial path and the result is true vacuously.  
Assume that  the lemma is true for some $p-1\geq 0$ (and all non-negative integers $\leq p-1$) and consider $\tau=(\tau_1,\ldots, \tau_{p})$ in $\Gamma(\mathcal Q_1)-N(\ast, m)$ of syllable length $p\geq 2$ and exponent sum equal to $0$. 
By symmetry, we may assume that $\tau_1$ is an $\{a,b\}$-path. Lemma \ref{Ptrade} (3) implies that there is a $k_1$-length $c^{-1}$ or $d^{-1}$-path $\gamma_1$ at $v_2$ and a $K$-path $\psi_1$ at $v_1$ (obtained by alternating the letters of of $\tau_1$ and $\gamma_1$) such that $\psi_1$ is homotopic to $(\tau_1,\gamma_1)$ by a homotopy in $\Gamma(\mathcal Q_1)-N(\ast,m)$ (see Figure \ref{FigL1}). Let $w_1$ be the end point of $\gamma_1$. We now must replace the path $(\gamma_1^{-1}, \tau_2,\ldots, \tau_p)$ by a $K$-path.  Notice that the labeling of this path has $0$ exponent sum.

 \medskip
 
\noindent (1) If $\tau_2$ is a $\{c,d\}$-path then $(\gamma_1^{-1},\tau_2)$ is a single syllable of $(\gamma_1^{-1}, \tau_2,\ldots, \tau_p)$. Our induction hypothesis implies $(\gamma_1^{-1} \tau_2,\ldots, \tau_p)$ is homotopic to a $K$-path $\psi_2$ by a homotopy in $\Gamma(\mathcal Q_1)-N(\ast, m)$. Combining this homotopy with the homotopy of $\psi_1$ to $(\tau_1, \gamma_1)$ (Figure \ref{FigL1}) gives a homotopy of $\tau$ to $(\psi_1,\psi_2)$ with image in $\Gamma(\mathcal Q_1)-N(\ast,m)$. 

\medskip

\noindent (2) If $\tau_2$ is an $\{a^{-1}, b^{-1}\}$ path of length $k_2\geq k_1$ then let $\tau_2=(\bar\tau_2,\hat \tau_2)$ where $\bar\tau_2$ has length $k_1$ (see Figure \ref{FigL2}). Lemma \ref{Ptrade} implies that the path with the same label as $\gamma_1$ at any vertex of $\tau_2$ moves geodesically away from $\ast$. Since $\gamma_1^{-1}$ is a $c$ or $d$-path, $(\gamma_1^{-1}, \bar \tau_2)$ is homotopic to a $K$-path $\psi_2$ at $w_1$ (obtained by alternating letters of $\gamma_1^{-1}$ and $\bar\tau_2$) by a homotopy in $\Gamma(\mathcal Q_1)-N(\ast,m)$. 
Our induction hypothesis implies that the $0$-exponent sum path $(\hat\tau_2, \tau_3,\ldots , \tau_p)$ (of syllable length $p-1$) is homotopic to a $K$-path $\psi_3$ by a homotopy in $\Gamma(\mathcal Q_1)-N(\ast,m)$. Combining homotopies, $(\tau_1,\ldots, \tau_p)$ is homotopic to the $K$-path $(\psi_1,\psi_2,\psi_3)$ by a homotopy in $\Gamma(\mathcal Q_1)-N(\ast, m)$ (see Figure \ref{FigL2}). 

\medskip

\begin{figure}
\vbox to 3in{\vspace {-2in} \hspace {.5in}
\hspace{-1.2 in}
\includegraphics[scale=1]{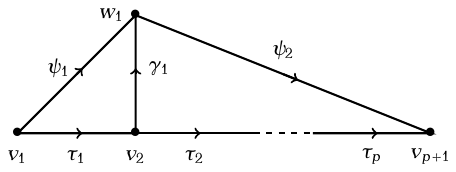}
\vss }
\vspace{-2.3in}
\caption{Part (1)} 
\label{FigL1}
\vspace{-.1in}
\end{figure}

\medskip

\begin{figure}
\vbox to 3in{\vspace {-1.5in} \hspace {.5in}
\hspace{-1.5 in}
\includegraphics[scale=1]{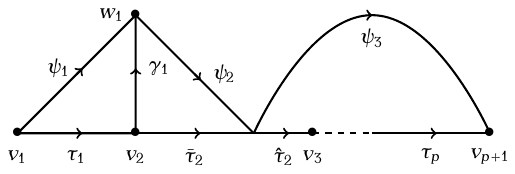}
\vss }
\vspace{-1.8in}
\caption{Part (2)} 
\label{FigL2}
\vspace{-.1in}
\end{figure}

\noindent (3) If $\tau_2$ is an $\{a^{-1}, b^{-1}\}$ path of length $k_2< k_1$ then Lemma \ref{Ptrade} (2) implies that $\tau_2$  can be extended by $\bar \tau_2$, a $(k_1-k_2)$-length $a^{-1}$ or $b^{-1}$-path in $\Gamma(\mathcal Q_1)-N(\ast,m)$. Lemma \ref{Ptrade} (3) implies that $(\gamma_1^{-1},\tau_2,  \bar\tau_2)$ is homotopic to a $K$-path $\psi_2$ by a homotopy in $\Gamma(\mathcal Q_1)-N(\ast,m)$ (see Figure \ref{FigL3}). Our induction hypothesis implies that the $0$-exponent path $(\bar\tau_2^{-1}, \tau_3,\ldots , \tau_p)$ is homotopic to a $K$-path $\psi_3$ by a homotopy in $\Gamma(\mathcal Q_1)-N(\ast,m)$. Combine homotopies, $(\tau_1,\ldots, \tau_p)$ is homotopic to the $K$-path $(\psi_1,\psi_2,\psi_3)$ by a homotopy in $\Gamma(\mathcal Q_1)-N(\ast, m)$.

\medskip

\noindent (4)  If $\tau_2$ is a $\{c^{-1}, d^{-1}\}$ path then let $(\hat\tau_2,\bar \tau_2)$ be the path obtained by eliminating backtracking in $(\gamma_1^{-1}, \tau_2)$ (then $(\gamma_1^{-1},\tau_2)$ is homotopic to $(\hat\tau_2, \bar\tau_2)$ in the image of $(\gamma_1^{-1},\tau_2)$). Let  $w_2$ be the end point of $\hat \tau_2$ (Figure \ref{FigL4.1}). The path $\hat\tau_2$ is a $c$ or $d$-path and $\bar\tau_2$ is a $\{c^{-1}, d^{-1}\}$ path. If either is trivial, then again we are left with a $0$-exponent word in $p-1$ syllables that (by our induction hypothesis) is homotopic to a $K$-path in $\Gamma-N(\ast,m)$ and we combine homotopies to finish.

\begin{figure}
\vbox to 3in{\vspace {-1.5in} \hspace {.5in}
\hspace{-1.5 in}
\includegraphics[scale=1]{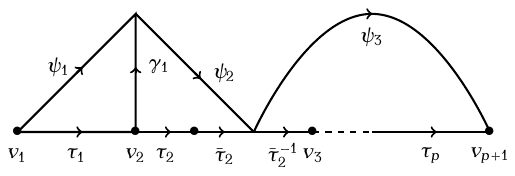}
\vss }
\vspace{-1.8in}
\caption{Part (3)} 
\label{FigL3}
\vspace{-.1in}
\end{figure}

\begin{figure}
\vbox to 3in{\vspace {-1.5in} \hspace {.5in}
\hspace{-1.5 in}
\includegraphics[scale=1]{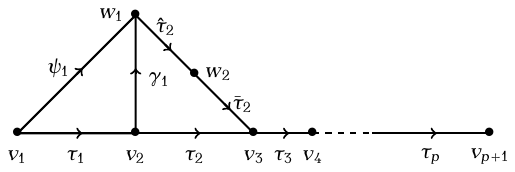}
\vss }
\vspace{-1.8in}
\caption{Part 4.1} 
\label{FigL4.1}
\vspace{-.1in}
\end{figure}

Otherwise, let $k=|\hat \tau_2|$. Lemma \ref{Ptrade} (3) implies that there is a $k$-length  $a^{-1}$ or $b^{-1}$ path $\bar \gamma_2$ at $w_2$ (that moves geodesically away from $\ast$) such that $(\hat \tau_2, \bar \gamma_2)$ is homotopic to a $K$-path $\psi_2$, by a homotopy in $\Gamma(\mathcal Q_1)-N(\ast, m)$. (See Figure \ref{FigL4.2}.)
\begin{figure}
\vbox to 3in{\vspace {-1.5in} \hspace {.5in}
\hspace{-1.5 in}
\includegraphics[scale=1]{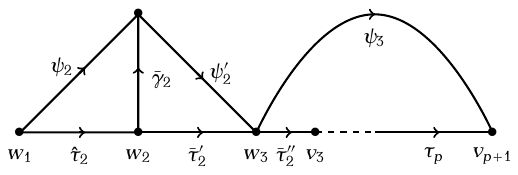}
\vss }
\vspace{-1.8in}
\caption{Part 4.2} 
\label{FigL4.2}
\vspace{-.1in}
\end{figure}
Let $\bar k=|\bar \tau_2|$. Recall, $|\bar \gamma_2|=|\hat\tau_2 |=k$. 

If $k\leq \bar k$ ($\bar\gamma_2$ is of shorter or the same length as $\bar \tau_2$) then write $\bar\tau_2=(\bar\tau_2',\bar \tau_2'')$ where $|\bar\tau_2'|=k=|\bar\gamma_2|$. Since $\bar\gamma_2$ moves geodesically away from $\ast$, Lemma \ref{Ptrade} (3) implies that the product homotopy for $\bar \gamma_2$ and $\bar \tau_2'$ has image in $\Gamma(\mathcal Q_1)- N(\ast, m)$. In particular $(\bar \gamma_2^{-1},\bar\tau_2')$ is homotopic to a $K$-path $\psi_2'$ (obtained by alternating letters of $\bar \gamma_2^{-1}$ and $\bar \tau_2'$) by a homotopy in $\Gamma(\mathcal Q_1)-N(\ast,m)$. 
Our induction hypothesis implies that the $0$-exponent sum path $(\bar\tau_2'' , \tau_3,\ldots, \tau_p)$ in $\Gamma(\mathcal Q_1)-N(\ast,m)$  (of syllable length $p-1$) is homotopic to a $K$-path $\psi_3$ by a homotopy in $\Gamma(\mathcal Q_1)-N(\ast,m)$. We combine the homotopies of Figures \ref{FigL4.1} and \ref{FigL4.2} along the path $(\hat \tau_2, (\bar \tau_2',\bar \tau_2'')=\bar \tau_2, \tau_3,\ldots, \tau_p)$ to finish. 

\begin{figure}
\vbox to 3in{\vspace {-1.5in} \hspace {.5in}
\hspace{-1.5 in}
\includegraphics[scale=1]{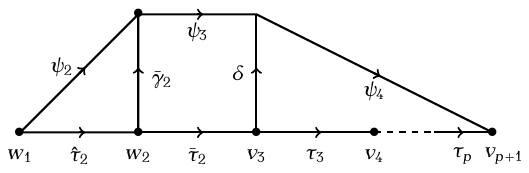}
\vss }
\vspace{-1.6in}
\caption{Part 4.3} 
\label{FigL4.3}
\vspace{-.1in}
\end{figure}

Finally, suppose $k>\bar k$ ($|\bar \gamma_2 |>|\bar \tau_2|$). Lemma \ref{Ptrade} (2) implies that either every $c^{-1}$ path at $v_3$ or every $d^{-1}$ path at $v_3$ moves geodesically away from $\ast$. Let $\delta$ be such a path of length $k-\bar k$. (See Figure \ref{FigL4.3}.) Since $\bar \gamma_2$ moves geodesically away from $\ast$, Lemma \ref{Ptrade}(3) implies the product homotopy for $\bar\gamma_2$ and $(\bar \tau_2, \delta)$ has image in $\Gamma(\mathcal Q_1)-N(\ast, m)$ and that the $K$-path $\psi_3$ (obtained by alternating letters of $\bar \gamma_2^{-1}$ and $(\bar\tau_2,\delta)$ at the end point of $\bar\gamma_2$) is homotopic to $(\bar \gamma_2^{-1},\bar \tau_2, \delta)$ in $\Gamma(\mathcal Q_1)-N(\ast, m)$. 
Our induction hypothesis implies that the $0$-exponent sum path $(\delta^{-1}, \tau_3,\ldots, \tau_p)$ in $\Gamma(\mathcal Q_1)-N(\ast,m)$  (of syllable length $p-1$) is homotopic to a $K$-path $\psi_4$ by a homotopy in $\Gamma(\mathcal Q_1)-N(\ast,m)$. Combining the homotopies of Figures \ref{FigL4.1} and \ref{FigL4.3} along $(\hat \tau_2, \bar \tau_2, \tau_3,\ldots ,\tau_p)$ finishes our proof. 
\end{proof}

\section {Stallings' Group Is Simply Connected At Infinity} \label{SCI}

\begin{proof} (of Theorem \ref{Main})
By Lemma \ref{Reduce} it is enough to prove that $\langle F(a,b)\times F(c,d)$ is simply connected at infinity in Stallings' group $S$. Let $C$ be a finite subcomplex of $X$. Let $H=\langle K\cup \{s\}\rangle$ and  $\Gamma(H)\subset X$ be the Cayley graph of $H$ with respect to $\{s,e_1,\ldots, e_q\}$. Recall that $\Gamma(K)\subset \Gamma (H)$ is the Cayley graph of $K$ with respect to $\{e_1,\ldots, e_q\}$. Let $\bar \Gamma(H)$ be $\Gamma(H)$ union all conjugation 2-cells for the relations $[s,e_i]$ for all $i$. Observe that $\bar \Gamma(H)$ is a subcomplex of $X$ that is homeomorphic to $\Gamma(K)\times \mathbb R$ and the action of the cyclic group $\langle s\rangle$ on $\bar \Gamma(H)$ respects this product structure.  
For $g\in S$, only finitely many $g\bar\Gamma(H)$ intersect $C$. List them as  $g_1\bar\Gamma(H), \ldots, g_m\bar\Gamma(H)$. Since $K$ is one ended (Lemma \ref{OneE}), Theorem \ref{sc} implies that $H$ is simply connected at infinity. There are finite complexes $D_1,\ldots, D_m$ in $X$ such that $\{s, e_1,\ldots, e_q\}$-loops in $g_i\bar\Gamma(H)-D_i$ are homotopically trivial in $X-C$ for all $i$. Let $D=C\cup (\cup_{i=1}^m D_i)$. Choose $k$ such that $D\cap \Gamma_1\subset N(\ast,k)(\subset \Gamma_1)$. We want to show that edge path loops in $\Gamma_1-N(\ast, k)$ are homotopically trivial in $X-C$. Let $\tau$ be an edge path loop in 
$$\Gamma_1-N(\ast,k)\subset \Gamma_1-D= \Gamma_1-(C\cup(\cup_{i-1}^m D_i))=(\Gamma_1-C)\cap(\cap_{i=1}^m\Gamma_1-D_i).$$ 
Lemma \ref{F2P} implies that $\tau$ is homotopic to an $K$-loop $\gamma$  by a homotopy in $\Gamma(\mathcal Q_1)-N(\ast, k)\subset \Gamma(\mathcal Q_1)-C$. Observe that $\gamma$ is homotopic to an $\{e_1,\ldots, e_q\}^{\pm 1}$ path $\gamma'$ by a homotopy that only uses ``triangle" 2-cells coming from the relations $e_i=w_i$. The vertices of $\gamma'$ are a subset of those of $\gamma$.  
Choose $g\in S$ such that $\gamma'$ has image in $g\Gamma(H)$. If $g\bar\Gamma(H)=g_i\bar\Gamma(H)$ for some $i\in \{1,\ldots, m\}$ then since $\gamma$ has image in $\Gamma_1-D_i$, $\gamma'$ has image in $g_i\bar\Gamma(H)-D_i$ and is homotopically trivial in $X-C$ (finishing our proof).

If $g\bar\Gamma(H)\ne g_i\bar\Gamma(H)$ for all $i$, then $g\bar\Gamma(H) \cap C=\emptyset$. Let $\bar N(\ast, j)$ be the $j$-neighborhood of $\ast$ in $X$. Choose an integer $j$ such that $\gamma'$ is homotopically trivial in $\bar N(v,j)$ for any vertex $v$ of $\gamma'$. Using the product structure of $\bar\Gamma(H)$, there is an integer $p>0$ such that $gs^pg^{-1} \gamma'$ has image in $g\bar \Gamma(H)-N^j(C)$ and so is homotopically trivial in $X-C$. The product structure of $\bar \Gamma(H)$ implies that $\gamma'$ is homotopic to $gs^pg^{-1}\gamma'$ in $g\bar \Gamma(H)$ (so this homotopy avoids $C$). Combining homotopies, $\gamma'$ is homotopically trivial by a homotopy in $X-C$.
\end{proof}

\bibliographystyle{amsalpha}
\bibliography{paper1}{}

\end{document}